\documentclass{amsart}
\usepackage{amssymb,latexsym}
\usepackage{a4}
\newtheorem{Thm}{Theorem}[section]
\newtheorem{Lemma}[Thm]{Lemma}
\newtheorem{Prop}[Thm]{Proposition}

\newtheorem{Rem}[Thm]{Remark}

\def\qed{~\hfill$\square$\medbreak}

\def\C{{\mathbb C}}
\def\N{{\mathbb N}}

\def\R{{\mathbb R}}

\def\cP{\mathcal P}

\def\SO{\mathrm{SO}}

\def\SL{\mathrm{SL}}
\def\SU{\mathrm{SU}}
\def\SP{\mathrm{Sp}}

\def\ab{{(\alpha,\beta)}}
\def\jacobi{P_n^\ab}
\def\gab{g_n^\ab}

\usepackage[usenames]{color}

\begin{document}
\title[Jacobi polynomials]
{Inequalities for Jacobi polynomials}
\author[Haagerup
and Schlichtkrull]{Uffe Haagerup
and Henrik Schlichtkrull}
\address{Department of Mathematical Sciences, University of Copenhagen, Denmark}
\email{haagerup@math.ku.dk, schlicht@math.ku.dk}
\date{January 30, 2012}
\begin{abstract}
A Bernstein type inequality is obtained for the Jacobi 
polynomials $\jacobi(x)$, which is uniform for  
all degrees $n\ge0$, all real $\alpha,\beta\ge0$,
and all values $x\in [-1,1]$.
It provides uniform bounds on a complete set of 
matrix coefficients for the irreducible representations of 
$\SU(2)$ with a decay of $d^{-1/4}$ in the dimension $d$ 
of the representation. Moreover it complements previous results 
of Krasikov on a conjecture of Erd\'elyi, Magnus and Nevai.
\end{abstract}
\maketitle

\section{Introduction}
For $\alpha,\beta\in\R$, $\alpha,\beta>-1$, 
and $n$ a non-negative integer
we denote by $\jacobi$ the Jacobi polynomial
with the standard normalization.
Recall that in terms of the Gauss
hypergeometric  function, 
$$\jacobi(x)= \frac{\Gamma(n+ \alpha+1)}{\Gamma(\alpha+1)\Gamma(n+1)} \,\,
{}_2F_1(-n,n+\alpha+\beta+1;\alpha+1;\frac{1-z}2).$$
Recall also that for a fixed pair $\ab$ these functions
are orthogonal polynomials on $[-1,1]$
for the weight function 
$$w^\ab(x)=(1-x)^\alpha(1+x)^\beta$$
with the explicit values
$$
\int_{-1}^1 \jacobi(x)^2\,w^\ab(x)\,dx
=\frac{2^{\alpha+\beta+1}}{2n+ \alpha+\beta+1}\,
\frac{\Gamma(n+\alpha+1)\Gamma(n+\beta+1)}
{\Gamma(n+1)\,\Gamma(n+\alpha+\beta+1)}
$$
(see \cite{Sze}, eq.~(4.3.3)).

For $x\in[-1,1]$ and $\alpha,\beta\ge 0$ let
$$\gab(x)=
\left(\frac{\Gamma(n+1)\,\Gamma(n+\alpha+\beta+1)}
{\Gamma(n+\alpha+1)\Gamma(n+\beta+1)}\right)^{1/2}
\left(\frac{1-x}2\right)^{\alpha/2}
\left(\frac{1+x}2\right)^{\beta/2}
\jacobi(x),
$$
then these functions are orthogonal on $[-1,1]$
for the constant weight. Moreover
\begin{equation}\label{Schur}
\frac12 \int_ {-1}^1 \gab(x)^2\,dx=
\frac1{2n+\alpha+\beta+1}.
\end{equation}
In suitable coordinates the functions $\gab$
with arbitrary non-negative integers 
$\alpha$, $\beta$ and $n$
comprise a natural and complete 
set of matrix coefficients for the 
irreducible representations of $\SU(2)$
(see Section \ref{Section repth} below).
The value $2n+\alpha+\beta+1$ in (\ref{Schur})
is exactly the dimension of
the corresponding irreducible representation.

We shall prove the following uniform upper bound

\begin{Thm}\label{thm1} 
There exists a constant $C>0$ such that
$$|(1-x^2)^{\frac14}\,\gab(x)|\le C(2n+\alpha+\beta+1)^{-\frac14}
$$
for all $x\in[-1,1]$, all $\alpha,\beta\ge 0$ 
and all non-negative integers $n$. 
\end{Thm}

We have not made a serious effort to find the best value of 
$C$, but at least our proof shows that $C<12$.

With standard normalization the inequality in 
Theorem \ref{thm1} amounts to the
following uniform bound for the Jacobi polynomials
\begin{equation}\label{Jacobi estimate}\begin{aligned}
&(\sin\theta )^{\alpha+\frac12}(\cos\theta)^{\beta+\frac12}
|\jacobi(\cos 2\theta)|\\
&\qquad\qquad \le \frac{C}{\sqrt2} (2n+\alpha+\beta+1)^{-\frac14}
\left(\frac{\Gamma(n+\alpha+1)\Gamma(n+\beta+1)}
{\Gamma(n+1)\Gamma(n+\alpha+\beta+1)}\right)^{1/2}
\end{aligned}\end{equation}
for $0\le\theta\le\pi/2$. 
The decay rate of $1/4$ in Theorem \ref{thm1} is optimal
when $\alpha$ and $\beta$ tend to infinity, see Remark
\ref{remark on Stirling}. However, 
if the pair $\ab$ is fixed, then $\jacobi(x)$ is
$O(n^{-1/2})$ for each $x\neq\pm1$, cf.~\cite{Sze}, Thm.~7.32.2. 
In particular, in 
Legendre's case $\alpha=\beta=0$ where $\jacobi(x)$
specializes to the Legendre polynomial $P_n(x)$,
the Bernstein inequality (refined by
Antonov and Kholshevnikov) 
\begin{equation}\label{Bern ineq}
(1-x^2)^{1/4} |P_n(x)| \le  (4/\pi)^{1/2} (2n+1)^{-1/2}, \qquad x\in[-1,1],
\end{equation}
is known to be sharp, see \cite{Sze}, Thm.~7.3.3,
and \cite{Lorch}.
We refer to \cite{Gautschi} for a further discussion
of the sharpest constant in (\ref{Jacobi estimate}),
with a subset of the current parameter range.

It is of interest also to express our inequality in terms of
the orthonormal polynomials defined by 
$$
\hat P_n^{\ab}(x) = 
\left(\frac{(2n+\alpha+\beta+1)\Gamma(n+1)\,\Gamma(n+\alpha+\beta+1)}
{2^{\alpha+\beta+1}\Gamma(n+\alpha+1)\Gamma(n+\beta+1)}\right)^{1/2}
\jacobi(x)$$
for which
$$
\int_{-1}^1 \hat P_n^{\ab}(x)^2 \,w^\ab(x)\,dx = 1.
$$
Here our estimate reads
$$(1-x^2)^{\frac14}\sqrt{w^{\alpha,\beta}(x)}|\hat P_n^{\ab}(x)|\le 
\frac{C}{\sqrt2}(2n+\alpha+\beta+1)^{\frac14}.
$$
The following generalization of Bernstein's inequality
(\ref{Bern ineq}) was conjectured by Erd\'elyi, Magnus and Nevai, 
\cite{NEM},
\begin{equation}\label{EMN-conjecture}
(1-x^2)^{\frac14}
\sqrt{w^{\alpha,\beta}(x)}|\hat P_n^{\ab}(x)|
\le C' (\alpha+\beta+2)^{1/4}
\end{equation}
for all $\alpha,\beta\ge -\frac12$ and all integers $n\ge 0$,
with a uniform constant $C'>0$. A stronger form of the 
conjecture was recently established by Krasikov,
\cite{Kras}, but only in the parameter range
$\alpha,\beta\ge \frac{1+\sqrt2}4$, $n\ge 6$. Our estimate
is valid for a more general range, but it involves
$2n+\alpha+\beta$ rather than $\alpha+\beta$.
Note however that by combining our results with those of 
\cite{Kras}, one can remove Krasikov's restriction $n\ge 6$ in the 
parameter range for the validity of (\ref{EMN-conjecture}).

The estimate (\ref{Jacobi estimate})
implies a similar estimate for the ultrasperical (Gegenbauer)
polynomials $C^{(\lambda)}_n(x)$, as these are directly
related to the Jacobi polynomials
$P^\ab_n(x)$ with $\alpha=\beta=\lambda-\frac12$.
Previous to \cite{Kras} this case had been considered
in \cite{KrasEMN}, and as above (\ref{Jacobi estimate}) allows the
removal of a restriction on the degree.

The proof of Theorem \ref{thm1} is based on an
expression for $\jacobi(x)$ as a contour integral, for which
we can estimate the integrand by elementary analysis.
The proof is simpler when $\alpha$ and $\beta$ are integers.
In this case, which is treated in Section \ref{Section integral},
the contour is just a circle. The general case is the discussed 
in Section \ref{Section general}.

\section{Motivation from representation theory}
\label{Section repth}

It is well known that the irreducible representations of
$\SU(2)$ can be expressed by Jacobi polynomials. 
In the physics literature it is customary to denote the 
corresponding matrix representations as {\it Wigner's d-matrices}.
We recall a few details (see \cite{Zelobenko}, \S 38,
\cite{Vilenkin}, Ch.~3, or \cite{Koornwindernotat}). 
The irreducible representations $\pi_l$
of $\SU(2)$ are parametrized by the non-negative
integers or half-integers $l=0,\frac12,1,\dots$, where
$2l+1$ is the corresponding dimension. The
standard representation space for $\pi_l$ is 
the space $\cP_l$ of polynomials
in two complex variables $z_1,z_2$, homogeneous of degree
$2l$, on which
the representation is given by
$$[\pi_l\begin{pmatrix} a&b\\c&d\end{pmatrix} f](z_1,z_2)
=f(az_1+cz_2,bz_1+dz_2).$$
Let
$$
k_\phi=\begin{pmatrix} e^{i\phi}&0
\\ 0& e^{-i\phi}
\end{pmatrix}\quad\text{and}\quad
t_\theta=\begin{pmatrix} \cos\theta &-\sin\theta 
\\ \sin\theta & \cos\theta
\end{pmatrix}
$$
for $\phi,\theta\in\R$, then every element $A\in\SU(2)$
allows a decomposition of the form
$A=k_\phi t_\theta k_{-\psi}$. 
The monomials $z_1^jz_2^k$ with $j+k=2l$ form a basis
for $\cP_l$, and it is convenient to use the notation
$$
h^l_p(z_1,z_2)=z_1^{l-p}z_2^{l+p}
$$
where $p=-l,-l+1,\dots,l$. Notice that
these are weight vectors 
$$\pi_l(k_\phi)h^l_p=e^{-i2p\phi} h^l_p,
\quad(p=-1,\dots,l).$$
Choosing the inner product on $\cP_l$ so that
$\pi_l$ is unitary, the functions $h^l_p$
form an orthogonal basis. We denote by
$\hat h^l_p$ the corresponding normalized basis
vectors. For $A\in\SU(2)$ the 
matrix elements $$m^l_{pq}(A)=\langle\pi_l(A)\hat h^l_q,
\hat h^l_p\rangle$$ with $p,q=-l,\dots,l$,
form the so-called Wigner's d-matrix.
Our result for the Jacobi polynomials
implies the following.

\begin{Thm}\label{thm2}
Let $C$ be the constant from Theorem \ref{thm1}. Then
\begin{equation}\label{matrix est}
|\sin 2\theta|^{1/2}\, |m^l_{pq}(k_\phi t_\theta k_{-\psi})|\le C
(2l+1)^{-1/4}
\end{equation}
for all $\phi,\theta,\psi\in\R$,
all $l=0,\frac12,1,\dots$
and all $p,q=-l,\dots,l$. Moreover, the exponent $1/4$ 
on the right hand side is best 
possible.
\end{Thm}

\begin{proof} 
Explicitly the matrix elements are given as follows
(see \cite{Zelobenko}, \cite{Vilenkin}, \cite{Koornwindernotat}). 
For  $p,q=-l,\dots,l$ such that $|q|\le p$,
$$
m^l_{pq}(k_\phi t_\theta k_{-\psi})=
e^{-i2p\phi} e^{i2q\psi}
g^\ab_n(\cos 2\theta),
$$
where
$$\alpha=p-q,\, \beta=p+q,\, n=l-p.$$
For other values of $p$ and $q$ there are similar expressions, 
and in all cases one has 
$$|m^l_{pq}(k_\phi t_\theta k_{-\psi})|=
|g^\ab_n(\cos 2\theta)|$$
where $\alpha=|p-q|$, $\beta=|p+q|$ and 
$n=l-\max\{|p|,|q|\}.$ Moreover
$$\dim\pi_l=2l+1=2n+\alpha+\beta+1.$$
Thus (\ref{matrix est}) follows directly from Theorem \ref{thm1}.
For the last statement of Theorem \ref{thm2},
see Remark \ref{remark on Stirling}.
\qed\end{proof}

\begin{Rem}
For $l$ integral $\pi_l$ descends to a representation of 
$\SO(3)$, and 
the matrix elements $m^l_{p0}$ with $q=0$ descend
to spherical harmonic functions on $S^2\simeq\SO(3)/\SO(2)$. 
With the common normalization from quantum mechanics
the spherical harmonics 
$Y^m_{l}$ with $-l\le m\le l$ 
satisfy
$$Y^m_l(\theta,\phi)=\pm\frac{(2l+1)^{1/2}}{(4\pi)^{1/2}}
\, g^{(\alpha,\alpha)}_{l-\alpha}(\cos\theta)\,e^{im\phi}, 
$$
where $\alpha=|m|$.
From Theorem \ref{thm1} we obtain the
uniform estimate
$$
|\sin\theta|^{1/2}\,|Y_{l}^m(\theta,\phi)|
\le \frac{C}{{(4\pi)}^{1/2}}(2l+1)^{1/4}
$$
for all $\theta,\phi$ and all integers $l,m$ with $|m|\le l$.
\end{Rem}

The Jacobi polynomials are also related to the harmonic 
analysis on the complex spheres with respect to the action
of the unitary group. The spherical functions for the pair 
$(U(q),U(q-1))$ are functions on the unit sphere in $\C^q$,
and in suitable coordinates they can be expressed by means of 
Jacobi functions $P^\ab_n$ with $\alpha=q-2$
(see \cite{Shapiro}, \cite{KII}). 
The direct motivation for the present 
paper was an application of this observation for $q=2$ to 
a study of $\SP(2,\R)$. In \cite{HdL} the first author and 
de Laat apply the uniform estimates of the present paper 
for the case $\alpha=0$, to show that $\SP(2,\R)$ does not 
have the approximation property. Earlier, Bernstein's 
inequality (\ref{Bern ineq}) had been used in \cite{LL} 
with a similar purpose for the group $\SL(3,\R)$.

\section{Integral parameters}\label{Section integral}

The proof is based on the following integral expression,
which is obtained by applying Cauchy's formula to
Rodrigues' formula for $\jacobi(x)$ (see \cite{Sze}, eq.~(4.3.1)),
\begin{equation}\label{integral formula}
(1-x)^\alpha(1+x)^\beta\,\jacobi(x)= (-\tfrac12)^n I^\ab_n(x)
\end{equation}
for $x\in(-1,1)$, where
\begin{equation}\label{In}
I^\ab_n(x)=\frac1{2\pi i} \int_{\gamma(x)} 
\frac{(1-z)^{n+\alpha}(1+z)^{n+\beta}}{(z-x)^n}\,\frac{dz}{z-x}.
\end{equation}
Here $\gamma(x)$ is any closed contour encircling $x$ in 
the positive direction. We assume in this section that
$\alpha$ and $\beta$ are integers $\ge 0$. 
Without this assumption
one would have to request also that $\gamma(x)$ does not
enclose the points $z=\pm 1$. We shall take $\gamma(x)=C(x,r)$,
the circle centered at $x$ and with a radius $r>0$ 
to be specified later. 

The case $n=0$ will be treated separately in Lemma \ref{n=0} below.
Here we assume $n\ge 1$ and let $a=\alpha/n$ and $b=\beta/n$, then
$$
\begin{aligned}
I^\ab_n(x)&=\frac1{2\pi i} \int_{C(x,r)} 
\left(\frac{(1-z)^{a+1}(1+z)^{b+1}}{z-x}\right)^n
\frac{dz}{z-x}\\
&=\frac1{2\pi i} \int_{C(0,r)} 
\left(\frac{(1-x-s)^{a+1}(1+x+s)^{b+1}}{s}\right)^n
\frac{ds}{s}.
\end{aligned}
$$

In order to select a suitable radius $r$
we look for the stationary points of the expression
inside the parentheses, as a function of $s$. We let 
$$\psi(s)=(a+1)\log(1-x-s)+(b+1)\log(1+x+s)-\log s$$
for $s\in\C$, and analyze the derivative
$$\psi'(s)=\frac{a+1}{s+x-1}+\frac {b+1}{s+x+1}-\frac1s,$$ 
which is independent of the branch 
cut used for the complex logarithm. Now
$$\psi'(s)=\frac{As^2+B(x)s+C(x)}{(s+x-1)(x+s+1)s}$$
where
$$A=a+b+1,\quad B(x)=(a+b)x+a-b,\quad C(x)=1-x^2.$$
The numerator is a second order polynomial in $s$ with the 
discriminant
$$\begin{aligned}
\Delta(x)&=B(x)^2-4AC(x)\\&=
(a+b+2)^2x^2+2(a^2-b^2)x+(a-b)^2-4(a+b+1),
\end{aligned}$$
which coincides with the polynomial $\Delta$
defined in \cite{CI}. The polynomial  $\Delta(x)$
has two real roots
$$
\begin{aligned} &x^+\\&x^-\end{aligned}
\biggr \}=
\frac{b^2-a^2\pm4\sqrt{(a+1)(b+1)(a+b+1)}}{(a+b+2)^2}
$$
for which $-1\le x^-< x^+\le 1$.
For $x^-<x<x^+$ we have $\Delta(x)<0$, 
and thus there are two conjugate
solutions $s=s_1,s_2$ to the equation $As^2+B(x)s+C(x)=0$.
They are 
$$s_1,s_2=\frac{-B(x)\pm i\sqrt{-\Delta(x)}}{2A}.$$
Note that
$$|s_1|^2=|s_2|^2=s_1s_2=\frac{C(x)}A=\frac{1-x^2}{a+b+1}.$$
Hence, if we choose the radius
\begin{equation}\label{defi r}
r=\sqrt{\frac{1-x^2}{a+b+1}},
\end{equation}
then our contour $C(0,r)$ will pass through the
stationary points of $\psi$.  
We define $r$ by (\ref{defi r})
for all $x\in(-1,1)$ (also when $\Delta(x)\ge 0$).

We now find
$$
|I^\ab_n(x)|\le \frac1{2\pi } \int_{0}^{2\pi} 
\left|(1-x-re^{i\theta})^{1+a}(1+x+re^{i\theta})^{1+b}
r^{-1}\right|^n
\,d\theta,
$$
and write
$$|(1-x-re^{i\theta})^{1+a}(1+x+re^{i\theta})^{1+b}r^{-1}|=e^{f( \cos\theta)}
$$
where 
\begin{equation}\label{defi f}
\begin{aligned}
f(t)=&\frac{a+1}2\ln\left(r^2+(1-x)^2-2r(1-x)t\right)\\
&+\frac{b+1}2\ln\left(r^2+(1+x)^2+2r(1+x)t\right)-\ln(r)
\end{aligned}
\end{equation}
for $t\in [-1,1]$.
Notice that we allow the possible value $f(t)=-\infty$
in the end points $t=\pm 1$. 
Let
\begin{equation}\label{t1,t2}
t_2=\frac{r^2+(1-x)^2}{2r(1-x)},\quad
t_1=-\frac{r^2+(1+x)^2}{2r(1+x)}
\end{equation}
then $t_1\leq -1$ and $1\leq t_2$. It follows that
\begin{equation}\label{eq f}
f(t)=\frac{a+1}2\ln(t_2-t)+
\frac{b+1}2\ln(t-t_1)+K
\end{equation}
where
\begin{equation}\label{eq K}
K=\frac{a+1}2\ln(1-x)+\frac{b+1}2\ln(1+x)
+\frac{a+b}2\ln r
+\frac{a+b+2}2\ln2
\end{equation}
is independent of $t$. With (\ref{eq f}) we can extend
the domain of definition for $f$ to $[t_1,t_2]\supset[-1,1]$.
For later reference we note that from (\ref{t1,t2}) and
(\ref{defi r}) it follows that
\begin{equation}\label{second t1,t2 expression}
t_1=\frac{-(a+b+2)-(a+b)x}{2\sqrt{a+b+1}\sqrt{1-x^2}},\qquad
t_2=\frac{(a+b+2)-(a+b)x}{2\sqrt{a+b+1}\sqrt{1-x^2}},
\end{equation}
and
\begin{equation}\label{t2-t1}
t_2-t_1=\frac{a+b+2}{\sqrt{a+b+1}\sqrt{1-x^2}}.
\end{equation}

We have
$$
|I^\ab_n(x)|\le \frac1{2\pi } \int_{0}^{2\pi} e^{nf(\cos\theta)}
\,d\theta.
$$
From (\ref{eq f}) we find
\begin{equation}\label{fprime}
f'(t)=-\frac{a+1}{2(t_2-t)}+
\frac{b+1}{2(t-t_1)}=\frac{(a+b+2)(t_0-t)}{2(t_2-t)(t-t_1)},
\end{equation}
where $t_0$ is the convex combination
\begin{equation}\label{defi t0}
t_0=\frac{(a+1)t_1+(b+1)t_2}{a+b+2}=
\frac{-a+b-(a+b)x}{2\sqrt{a+b+1}\sqrt{1-x^2}}\in(t_1,t_2).
\end{equation}
Moreover
$$f''(t)=-\frac{a+1}{2(t_2-t)^2}-\frac{b+1}{2(t-t_1)^2}<0.
$$
Hence the function $f(t)$ is concave and
has a global maximum at $t_0$. 
We thus obtain the initial estimate
\begin{equation}\label{initial}
|I^\ab_n(x)|\le \frac1{\pi } \int_{0}^{\pi} e^{nf(\cos\theta)}
\,d\theta\le e^{nf(t_0)}.
\end{equation}
Since
\begin{equation}\label{t0 differences}
t_2-t_0=\frac{(a+1)(t_2-t_1)}{a+b+2},\quad
t_0-t_1=\frac{(b+1)(t_2-t_1)}{a+b+2}
\end{equation}
we find
$$
f(t_0)=\frac{a+1}2\ln\frac{(a+1)(t_2-t_1)}{a+b+2}
+\frac{b+1}2\ln\frac{(b+1)(t_2-t_1)}{a+b+2}+K,
$$
and from (\ref{eq K}) and (\ref{t2-t1})
it then follows that
$$
f(t_0)=\frac12
\ln\left(
\frac{2^{a+b+2}(a+1)^{a+1}(b+1)^{b+1}}{(a+b+1)^{a+b+1}}
(1-x)^a(1+x)^b\right).
$$
Thus 
$$\begin{aligned}
e^{nf(t_0)}&\le 
\left(\frac{2^{a+b+2}(a+1)^{a+1}(b+1)^{b+1}}{(a+b+1)^{a+b+1}}
(1-x)^a(1+x)^b\right)^{n/2}\\
&=
\left(\frac{2^{a+b+2}(a+1)^{a+1}(b+1)^{b+1}}{(a+b+1)^{a+b+1}}\right)^{n/2}
({1-x})^{\alpha/2}({1+x})^{\beta/2}.
\end{aligned}$$
The inequality
\begin{equation}\label{estimate}
\begin{split}
\frac{\Gamma(n+1)\Gamma(n+\alpha+\beta+1)}
{\Gamma(n+\alpha+1)\Gamma(n+\beta+1)}&\,
\left(\frac{(a+1)^{a+1}(b+1)^{b+1}}{(a+b+1)^{a+b+1}}\right)^{n}\\
&\quad\le 
\left(\frac{(n+1)(n+\alpha+\beta+1)}{(n+\alpha+1)(n+\beta+1)}\right)^{1/2}
\end{split}
\end{equation}
will be shown in Lemma \ref{Stirling bound}. 
Inserting (\ref{initial}) and (\ref{estimate}) into our definition
of $g_n^\ab$ we obtain the initial bound
\begin{equation}\label{initial estimate}
|g_n^\ab(x)|\le 
\left(\frac{(n+1)(n+\alpha+\beta+1)}{(n+\alpha+1)(n+\beta+1)}\right)^{1/4}.
\end{equation}
In particular, since $(n+1)(n+\alpha+\beta+1)\le(n+\alpha+1)(n+\beta+1)$
it follows that $|g_n^\ab(x)|\le 1$
(which could also be seen directly from the fact that
$g_n^\ab$ is a unitary matrix coefficient of orthonormal
vectors).

In order to improve the estimate we need
to replace the inequality $f(t)\le f(t_0)$  
by a stronger inequality. In Proposition
\ref{lemma bound f} below we shall establish
the inequality
\begin{equation}\label{bound f}
f(t)\leq f(t_0)+\frac D{1+t_0^2} f''(t_0)(t-t_0)^2
\end{equation}
for $t\in[-1,1]$, with a suitable constant $D>0$.  
Following the argument from before
and taking into account the second term in (\ref{bound f})
we can then improve (\ref{initial}) with the extra factor
$$\frac1{\pi}\int_0^{\pi} \exp\left(\frac{nD}{1+t_0^2} 
f''(t_0)(\cos\theta-t_0)^2
\right)\,d\theta
$$
on the right hand side. 

For the estimation of the exponential 
integral we use Lemma \ref{bound on exp integral}
below, which is applicable 
since $f''(t_0)<0$. We let 
$$
u=t_0\sqrt{\frac{nD}{1+t_0^2}|f''(t_0)|},
\quad
v=\sqrt{\frac{nD}{1+t_0^2}|f''(t_0)|},
$$
and observe that $u^2+v^2=nD|f''(t_0)|$.
We thus obtain 
\begin{equation}\label{I ineq}
|I^\ab_n(x)|\le  2 e^{nf(t_0)} \left(nD|f''(t_0)|\right)^{-1/4}
\end{equation}
and hence (\ref{initial estimate}) has been improved to
$$
|g_n^\ab(x)|\le 
\left(\frac{(n+1)(n+\alpha+\beta+1)}{(n+\alpha+1)(n+\beta+1)}\right)^{1/4}
2\left(nD|f''(t_0)|\right)^{-1/4} .
$$
From (\ref{fprime}), (\ref{t0 differences}) and (\ref{t2-t1})
it follows that
\begin{equation}\label{eq f''}
f''(t_0)=-\frac{a+b+2}{2(t_0-t_1)(t_2-t_0)}=
-\frac{(a+b+1)(a+b+2)}{2(a+1)(b+1)}(1-x^2),
\end{equation}
and hence
$$
|f''(t_0)|=
\frac{(\alpha+\beta+n)(\alpha+\beta+2n)}{2(\alpha+n)(\beta+n)}
(1-x^2).
$$
Since
$$
\frac{n+\alpha+\beta+1}{(n+\alpha+1)(n+\beta+1)}
\le 
\frac{n+\alpha+\beta}{(n+\alpha)(n+\beta)}
$$
and
$$
\frac{n+1}{n(2n+\alpha+\beta)}\le \frac{3}{2n+\alpha+\beta+1}
$$
for all $n\ge 1$ and $\alpha,\beta\ge 0$,
it finally follows that
$$
|g_n^\ab(x)|\le C'(\alpha+\beta+2n+1)^{-1/4}(1-x^2)^{-1/4}
$$
where $C'=2\sqrt[4]{6/D}=2\sqrt[4]{168}<8$
with the value $D=1/28$ from below. 
This completes the proof of Theorem
\ref{thm1} in the integral case
(up to the cited results from below).\qed

\begin{Prop}\label{lemma bound f}
Fix $x\in[-1,1]$ and let $f(t)$ and $t_0$ be as above. Then
\begin{equation*}
f(t)\leq f(t_0)+\frac1{28(1+t_0^2)} f''(t_0)(t-t_0)^2
\end{equation*}
for all $t\in[-1,1]$.
\end{Prop}

\begin{proof}
We begin the proof by a sequence of lemmas. 

\begin{Lemma}\label{lemma t0} 
The following relation holds
\begin{equation}\label{t0 equation}
(a+b)^2+4(a+b+1)t_0^2=\frac{2a^2}{1-x}+\frac{2b^2}{1+x}.
\end{equation}
\end{Lemma}

\begin{proof}
Using (\ref{defi t0}) we obtain 
$$4(a+b+1)t_0^2= \frac{(a-b+(a+b)x)^2}{1-x^2}.$$
On the other hand
$$\frac{2a^2}{1-x}+\frac{2b^2}{1+x}
=\frac{2(a^2+b^2+(a^2-b^2)x)}{1-x^2}.$$ 
Hence (\ref{t0 equation}) follows from the identity
$$
(a+b)^2(1-x^2)+(a-b+(a+b)x)^2=2(a^2+b^2+(a^2-b^2)x),
$$
which is straightforward.\qed
\end{proof}

\begin{Lemma}\label{l2} 
We have
$$1-x^2\le 16\frac{(a+1)(b+1)}{(a+b+2)^2}(1+t_0^2)$$
for all $x\in[-1,1]$.
\end{Lemma}

\begin{proof}
Note first that if we replace the triple $(a,b,x)$ by 
$(b,a,-x)$, then $t_1,t_0,t_2$ are replaced by 
$-t_2,-t_0,-t_1$ and hence the asserted inequality
is unchanged. We may thus assume that $a\le b$.

It follows from Lemma \ref{lemma t0} that
$$(a+b)^2+4(a+b+1)t_0^2 \ge \frac{2b^2}{1+x}$$
and therefore
$$1+x\ge \frac{2b^2}{(a+b)^2+4(a+b+1)t_0^2}.$$
Hence
$$1-x\le 2-
 \frac{2b^2}{(a+b)^2+4(a+b+1)t_0^2}
=
 2\frac{a^2+2ab+4(a+b+1)t_0^2}{(a+b)^2+4(a+b+1)t_0^2}.$$
and
$$
1-x^2\le 2(1-x)\le 4
\frac{a^2+2ab+4(a+b+1)t_0^2}{(a+b)^2+4(a+b+1)t_0^2}.$$
Since the right hand side is an increasing
function of $t_0^2$ we have for
$t_0^2\le 1$ that
$$
1-x^2\le 4\frac{a^2+2ab+4(a+b+1)}{(a+b)^2+4(a+b+1)}
\le 16\frac{(a+1)(b+1)}{(a+b+2)^2},$$
where in the last step we used that $a\le b$ implies
$a^2+2ab\le 4ab$.
For $t_0^2\ge 1$ we obtain similarly
$$
1-x^2\le 4\frac{(a^2+2ab)t_0^2+4(a+b+1)t_0^2}{(a+b)^2+4(a+b+1)}
\le 16\frac{(a+1)(b+1)}{(a+b+2)^2}t_0^2.$$
This completes the proof of Lemma \ref{l2}.\qed\end{proof}

\begin{Lemma}\label{l3} 
We have
\begin{equation}\label{l3-ineq}
t_2-t_0\ge \frac{1}{4(1+t_0^2)^{1/2}} 
\qquad\text{and}\qquad 
t_0-t_1\ge \frac{1}{4(1+t_0^2)^{1/2}}.
\end{equation}
\end{Lemma}

\begin{proof} 
It follows from (\ref{t2-t1}) and Lemma \ref{l2} that
$$t_2-t_1
\ge \frac{(a+b+2)^2}{4\sqrt{(a+1)(b+1)(a+b+1)}}
(1+t_0^2)^{-1/2},
$$
and hence by (\ref{t0 differences})
$$
t_2-t_0\ge 
\frac{\sqrt{a+1}(a+b+2)}{4\sqrt{(b+1)(a+b+1)}}
(1+t_0^2)^{-1/2}.
$$
Using $(b+1)(a+b+1)\le (a+b+2)^2$ and $\sqrt{a+1}\ge 1$
we obtain the first 
inequality in (\ref{l3-ineq}). The second one is
analogous.\qed
\end{proof}

\begin{Lemma}\label{l4} 
We have
\begin{equation}\label{l4 ineq}
(u-t_1)(t_2-u)\le 14 (1+t_0^2)(t_0-t_1)(t_2-t_0)
\end{equation}
for all $u\in [t_1,t_2]$ for which
$-1\le u\le t_0$ or $t_0\le u\le 1$.
\end{Lemma}

\begin{proof}
We first assume
$a\le b$.
Then by (\ref{t0 differences})
\begin{equation}\label{one}
u-t_1\le t_2-t_1=\frac{a+b+2}{b+1}(t_0-t_1)\le 2(t_0-t_1).
\end{equation}
In order to estimate $t_2-u$ we first note that
$|u-t_0|\le 1+|t_0|$
and hence
$$t_2-u\le t_2-t_0+|t_0-u|\le t_2-t_0+1+|t_0|.$$
By Lemma \ref{l3}
$$
1+|t_0|\le \sqrt{2}(1+t_0^2)^{1/2}\le 
4\sqrt{2}(1+t_0^2)(t_2-t_0)
$$
and hence
\begin{equation}\label{two}
t_2-u\le (1+4\sqrt2)(1+t_0^2)(t_2-t_0)
\le 7(1+t_0^2)(t_2-t_0).
\end{equation}
Now (\ref{one}) and (\ref{two}) together
imply (\ref{l4 ineq}). The proof for $a\ge b$ is 
analogous.\qed
\end{proof}

We can now prove Proposition \ref{lemma bound f}. Let $t\in[-1,1]$.
It follows from (\ref{fprime}),  (\ref{l4 ineq}) 
and (\ref{eq f''}) that 
$$
\begin{aligned}
\frac{f'(u)}{u-t_0}&=-\frac{a+b+2}{2(u-t_1)(t_2-u)}\\
&\le -\frac{a+b+2}{28(1+t_0^2)(t_0-t_1)(t_2-t_0)}=
 \frac{f''(t_0)}{14(1+t_0^2)}
\end{aligned}
$$
for all $u\in\R$ between $t$ and $t_0$.
Hence
$$
\begin{aligned}
f(t)&=f(t_0)+\int_{t_0}^t f'(u)\,du\\&\le
f(t_0)+ \frac{f''(t_0)}{14(1+t_0^2)}
\int_{t_0}^t (u-t_0)\,du=
f(t_0)+\frac{f''(t_0)}{28(1+t_0^2)}(t-t_0)^2.
\qquad\square
\end{aligned}
$$
\end{proof}

\begin{Lemma}\label{bound on exp integral}
Let $u,v\in\R$ with $u^2+v^2>0$. Then  
\begin{equation}\label{from Tim}
\frac1{\pi}
\int_0^{\pi} e^{-(u+v\cos s)^2}\,ds\leq \frac2{(u^2+v^2)^{1/4}}
\end{equation}
\end{Lemma}

\smallskip
\begin{proof} 
We will show (\ref{from Tim}) with the slightly stronger bound
$$\frac{\sqrt2}{\sqrt{\max\{|u|,|v|\}}}.$$
The statement is invariant
under the map $(u,v) \mapsto (-u,-v)$ and, using the substitution 
$s \mapsto \pi -s$, also under $v \mapsto -v$. 
Hence, it is sufficient to show
$$
\frac1{\pi}
\int_0^{\pi} e^{-(u-v\cos s)^2}\,ds\leq \frac{\sqrt2}{\sqrt{\max\{u,v\}}}
$$
for $u \geq 0$, $v \geq 0$. 

Suppose first $0 \leq u \leq v$, then  $v\neq0$.
Let $\sigma \in [0,\frac{\pi}{2}]$ be such that
$\cos{\sigma}=\frac{u}{v}$. Then
$$u-v\cos{s}=v(\cos\sigma-\cos s)=
2v\sin(\frac{s+\sigma}{2})\sin(\frac{s-\sigma}{2}).$$ 
Note that
$\sin(\frac{s+\sigma}{2}) \geq |\sin(\frac{s-\sigma}{2})|$ 
because
$\sin^2(\frac{s+\sigma}{2})-\sin^2(\frac{s-\sigma}{2})
=\sin s\sin\sigma \geq 0$
for 
$s \in [0,\pi]$
and $\sigma \in [0,\frac{\pi}{2}]$.
Using
also that $|\sin t| \geq \frac{2}{\pi}|t|$ for $|t| \leq \frac{\pi}{2}$, 
it
follows that
\begin{equation} \nonumber
\begin{split}
\frac{1}{\pi} \int_{0}^{\pi} e^{-(u-v\cos{s})^2} ds &= \frac{1}{\pi}
\int_{0}^{\pi} e^{-4v^2 \sin^2(\frac{s+\sigma}{2})\sin^2(\frac{s-\sigma}{2})} ds
\\
&\leq \frac{1}{\pi} \int_{0}^{\pi} e^{-4v^2\pi^{-4}(s-\sigma)^4}ds \\
&\leq \frac{1}{\pi} \int_{-\infty}^{\infty} e^{-4v^2\pi^{-4}s^4}ds
\le \frac{2}{\sqrt{2v}} ,
\end{split}
\end{equation}
where we used that 
$\int_0^{\infty}e^{-t^4}\,dt=\Gamma(\frac{5}{4})\leq1$.

Suppose next $0 \leq v \leq u \leq 2v$. Then $u-v\cos{s} \geq
v(1-\cos{s}) = 2v\sin^2(\frac{s}{2})$. Hence,
\begin{equation} \nonumber
\begin{split}
\frac{1}{\pi} \int_{0}^{\pi} e^{-(u-v\cos{s})^2}ds &\leq \frac{1}{\pi}
\int_{0}^{\pi} e^{-4v^2\sin^4(\frac{s}{2})} ds \\
&\leq \frac{1}{\pi} \int_{0}^{\pi} e^{-4v^2\pi^{-4}s^4} ds
\leq \frac{1}{\sqrt{2v}}\leq\frac1{\sqrt{u}}
\end{split}
\end{equation}
using again $\int_0^{\infty}e^{-t^4}\,dt\leq1$.

Suppose finally $0 \leq 2v \leq u$. Then $u-v\cos{s} \geq \frac u2$ and
hence
\begin{equation*} 
\frac{1}{\pi} \int_{0}^{\pi} e^{-(u-v\cos{s})^2}ds
\leq e^{-\frac{u^2}{4}} \leq \frac1{\sqrt{u}}
\end{equation*}
where we used that $xe^{-x^4}\leq \frac1{\sqrt2}$ for all $x\ge 0$.
\qed\end{proof}

\section{Some inequalities with gamma functions}\label{gamma inequalities}

In this section we prove some inequalities which were used in the preceding
section. We assume that $\alpha,\beta$ are real and non-negative.

\begin{Lemma}\label{Stirling bound}
Let $n, \alpha,\beta\ge 0$. Then
\begin{equation}\label{Stirling ineq}
\begin{split}
&\frac{\Gamma(n+1)\Gamma(n+\alpha+\beta+1)}
{\Gamma(n+\alpha+1)\Gamma(n+\beta+1)}\\
&\qquad\qquad\le
\frac{n^n(\alpha+\beta+n)^{\alpha+\beta+n}}
{(\alpha+n)^{\alpha+n}(\beta+n)^{\beta+n}}\,
\left(\frac{(n+1)(n+\alpha+\beta+1)}{(n+\alpha+1)(n+\beta+1)}\right)^{1/2}.
\end{split}
\end{equation}
\end{Lemma}

\begin{proof} 
We have for $x,y,z\ge 0$ 
\begin{equation}\label{double integral}
\ln \frac{\Gamma(x+1)\Gamma(x+y+z+1)}{\Gamma(x+y+1)\Gamma(x+z+1)}
=
\int_0^y\int_0^z (\ln\Gamma)''(x+s+t+1)\,dt\,ds.
\end{equation}
We claim that
\begin{equation}\label{Gamma ineq}
(\ln\Gamma)''(u+1)\le \frac1u-\frac1{2(u+1)^2}
\end{equation}
for all $u>0$. The asserted inequality (\ref{Stirling ineq})
follows easily from (\ref{double integral}) and (\ref{Gamma ineq}).

In order to prove (\ref{Gamma ineq}) we recall that
$$
(\ln\Gamma)''(u+1)=\sum_{k=1}^{\infty} \frac1{(u+k)^2}
=\sum_{k=0}^{\infty} A(u+k),
$$
where $$A(u)=\frac1{(u+1)^2}.$$
For the other side of (\ref{Gamma ineq}) 
we use the telescoping series 
$$
\frac1u=\sum_{k=0}^{\infty} B(u+k),
\qquad\qquad
\frac1{2(u+1)^2}=\sum_{k=0}^{\infty} C(u+k),
$$
where 
$$B(u)=\frac1{u}-\frac1{u+1}=\frac1{u(u+1)}$$
and
$$C(u)=\frac1{2(u+1)^2}-\frac1{2(u+2)^2}=\frac{2u+3}{2(u+1)^2(u+2)^2}.$$
We observe that
$$C(u)\le \frac1{(u+1)^2(u+2)}$$
and hence
$$
B(u)-C(u)\ge \frac1{u(u+1)}-\frac1{(u+1)^2(u+2)}=
\frac{u^2+2u+2}{u(u+1)^2(u+2)}\ge A(u).
$$
We obtain (\ref{Gamma ineq}) by termwise application of this
inequality to the series.\qed
\end{proof}

\begin{Lemma}\label{new lemma 4.2}
For $\alpha,\beta\ge0$ 
$$
\frac{\Gamma(\alpha+\beta+1)}{\Gamma(\alpha+1)\Gamma(\beta+1)}
\le
\frac{(\alpha+\beta+\frac12)^{\alpha+\beta+\frac12}(\frac12)^{\frac12}}
{(\alpha+\frac12)^{\alpha+ \frac12}(\beta+\frac12)^{\beta+\frac12}}.
$$
\end{Lemma}
\begin{proof}
Following the preceding proof one deduces this inequality 
from 
\begin{equation*}
(\ln\Gamma)''(u+1)\le \frac1{u+\frac12}.
\end{equation*}
The latter inequality is also seen as in
the preceding proof, by using the telescoping series
$$\frac1{u+\frac12}=\sum_{k=0}^\infty D(u+k)$$
where
$$D(u)=\frac1{u+\frac12}-\frac1{u+\frac32}=\frac1{(u+\frac12)(u+\frac32)}
\ge\frac1{(u+1)^2}=A(u).
$$
\qed\end{proof}

\begin{Lemma}\label{n=0}
Let $\alpha,\beta\ge 0$ and $-1\le x\le 1$. Then
$$0\le (1-x^2)^{1/4}g^\ab_0(x)\le (\alpha+\beta+1)^{-1/4}.$$
\end{Lemma}

\begin{proof} 
Since $P^\ab_0(x)=1$, we have $g^\ab_0(x)\ge 0$ and
$$
(1-x^2)^{\frac12}g^\ab_0(x)^2=
\frac{2\Gamma(\alpha+\beta+1)}{\Gamma(\alpha+1)\Gamma(\beta+1)}
\left(\frac{1-x}2\right)^{\alpha+\frac12}
\left(\frac{1+x}2\right)^{\beta+\frac12}.
$$
For $\mu,\nu\ge 0$ the function $\varphi(x)=(1-x)^\mu(1+x)^\nu$
on $[-1,1]$ satisfies
$$\max_{x\in[-1,1]}\varphi(x)=\varphi\left(\frac{\nu-\mu}{\nu+\mu}\right)=
\frac{2^{\mu+\nu}\mu^\mu\nu^\nu}{(\mu+\nu)^{\mu+\nu}}.$$
Hence by Lemma \ref{new lemma 4.2}
\begin{equation}\label{max val}
\begin{aligned}
\max_{x\in[-1,1]} (1-x^2)^{\frac12}g^\ab_0(x)^2&=
\frac{2\Gamma(\alpha+\beta+1)}{\Gamma(\alpha+1)\Gamma(\beta+1)}
\frac{(\alpha+\frac12)^{\alpha+\frac12}(\beta+\frac12)^{\beta+\frac12}}
{(\alpha+\beta+1)^{\alpha+\beta+1}}\\
&\le h(\alpha+\beta) (\alpha+\beta+1)^{-1/2},
\end{aligned}
\end{equation}
where
$$
h(t)=\sqrt2 \left(\frac{t+\frac12}{t+1}\right)^{t+\frac12}.
$$
Since 
$$
(\log h)'(t)=\frac1{2(t+1)}+\log\left(\frac{t+\frac12}{t+1}\right)
=\int_{t+\frac12}^{t+1} (\frac1{t+1}-\frac1u)\,du \le 0$$
it follows that $h(t)\le h(0)=1$ for all $t\ge 0$.
This proves lemma \ref{n=0}.\qed
\end{proof}

\begin{Rem}\label{remark on Stirling} 
It follows from {\rm (\ref{max val})} and
Stirling's formula that 
$$\max (1-x^2)^{1/4}|g^\ab_0(x)| 
\sim
(2/\pi)^{1/4}(\alpha+\beta+1)^{-1/4}
$$
when $\alpha\to\infty$ and $\beta\to\infty$. Hence
the decay rate $1/4$ in Theorem \ref{thm1} cannot be
improved. This was observed already in \cite{NEM}, p.~604.

In this connection it can be noted that for each 
$l=0,\frac12,1,\dots$, the irreducible representation 
$\pi_l$ of $\SU(2)$ will exhibit matrix coefficients 
in which the functions $g^\ab_0$ for $\alpha+\beta=2l$
occur (see Section \ref{Section repth}). 
In particular, it follows that a positive solution to 
the EMN-conjecture mentioned in the introduction, will not 
significantly improve the representation theoretic
content of Theorem \ref{thm1}, discussed in 
Section~\ref{Section repth}.
\end{Rem}

\section{The general case}\label{Section general}

In this section $n\in\N_0$ and $\alpha, \beta$ are non-negative
real numbers. We have already proved in Lemma \ref{n=0} that
$$
|g^\ab_0(x)|\le  (\alpha+\beta+1)^{-1/4}, \quad
x\in[-1,1],\,\alpha,\beta\ge 0,
$$
so we can assume that $n>0$. As in Section \ref{Section integral}
we put $a=\alpha/n$ and $b=\beta/n$ and use the integral representation 
(\ref{integral formula})-(\ref{In}) of $\jacobi(x)$, with a
closed contour $\gamma(x)$ encircling $x$ in the positive
direction. In addition we assume now that $\gamma(x)$ does not
intersect the branch cuts $]-\infty,-1]$ and $[1,\infty[$.
As before we define $r>0$ by (\ref{defi r}) and consider
the circle $C(x,r)$. For $|x|<1$ we find
$$
1<x+r\quad\Leftrightarrow\quad x>\frac{a+b}{a+b+2},
$$
and consequently
$$
-1>x-r\quad\Leftrightarrow\quad x<-\frac{a+b}{a+b+2}.
$$
Hence we can distinguish the following cases:

\begin{description}
\item[Case 1] $\frac{a+b}{a+b+2}<x<1$. Then $1$ is inside 
and $-1$ is outside $C(x,r)$. 
\item[Case 2] $|x|<\frac{a+b}{a+b+2}$. Both $1$ and $-1$ 
are outside $C(x,r)$. 
\item[Case 3] $-1<x<-\frac{a+b}{a+b+2}$. Here $1$ is outside 
and $-1$ is inside $C(x,r)$. 
\end{description}

By continuity it suffices to prove Theorem \ref{thm1} in each of these
three cases. As the proof given in Section \ref{Section integral}
is valid without modification in Case 2, we need only consider the 
other two cases. Note that the integral
$$
J^\ab_n(x):=\frac1{2\pi i} \int_{C(x,r)} 
\frac{(1-z)^{n+\alpha}(1+z)^{n+\beta}}{(z-x)^{n+1}}\,dz
$$
makes sense for all $\alpha,\beta\ge 0$, although the argument of
the integrand
may become discontinuous at $z=x+r$ or at $z=x-r$ when these
points belong to the branch cuts. As in Section~\ref{Section integral},
see (\ref{initial}),
$$
|J^\ab_n(x)|\le \frac1{\pi} \int_{0}^{\pi}
e^{n f(\cos\theta)}\,d\theta
$$
where $f$ is the function defined by (\ref{defi f}).
Note that $f$ depends on $a$, $b$ and $x$. When necessary we 
denote it by $f=f_{a,b,x}$.

\begin{Lemma}\label{remainder estimates} 
The integral {\rm (\ref{In})} satisfies
\begin{equation}\label{I=J+R}
I^\ab_n(x)=J^\ab_n(x)+R^\ab_n(x)
\end{equation}
where
$|R^\ab_n(x)|\le e^{n f(1)}$ in Case 1, 
$R^\ab_n(x)=0$ in Case 2, and
$|R^\ab_n(x)|\le e^{n f(-1)}$ in Case 3.
\end{Lemma}

\begin{proof} Consider first Case 1, and note that
$$f(1)=\ln\left((r-1+x)^{a+1}(r+1+x)^{b+1}r^{-1}\right).$$
We let the closed contour
$\gamma(x)$ follow $C(x,r)$ except for a small arc
around the possible locus of 
discontinuity at $x+r$. Let $\delta>0$ be such that
the removed arc consist of points $z_1+iz_2$ in the strip
$|z_2|<\delta$.
The end points below and above $x+r$
are joined to $1\pm i\delta$ by line segments  
along the axis. Finally $1-i\delta$ and $1+i\delta$
are connected by a half circle crossing the axis to the left of 1.
In the limit $\delta\to 0^+$ we obtain (\ref{I=J+R})
with
$$
\begin{aligned}
R^\ab_n(x)&=-\frac{\sin(\pi(n+\alpha))}\pi \int_1^{x+r} 
\frac{(z-1)^{n+\alpha}(1+z)^{n+\beta}}{(z-x)^{n+1}}\,dz\\
&=(-1)^{n-1}\frac{\sin(\pi\alpha)}\pi \int_{1-x}^{r} 
\frac{(s+x-1)^{n+\alpha}(1+s+x)^{n+\beta}}{s^{n+1}}\,ds.
\end{aligned}
$$
In particular, $R^\ab_n(x)=0$ if $\alpha=0$
so that we may assume $\alpha>0$.
For $x<1$ and $0<s<r$ we have $\frac sr(1-x)\le 1-x$ and hence
$s+x-1\le \frac{s}{r}(r+x-1)$.  
It follows that
$$
\frac{(s+x-1)^{n+\alpha}(1+s+x)^{n+\beta}}{s^{n+1}}
\le
\frac{(r+x-1)^{n+\alpha}(1+r+x)^{n+\beta}s^{\alpha-1}}{r^{n+\alpha}}
$$
for $0<1-x<s<r$. Thus
$$
\begin{aligned}
|R^\ab_n(x)|&\le \frac{|\sin(\pi\alpha)|}\pi 
\frac{(r+x-1)^{n+\alpha}(1+r+x)^{n+\beta}}{r^{n+\alpha}}
\int_{0}^r s^{\alpha-1}\,ds\\
&=\frac{|\sin(\pi\alpha)|}{\pi\alpha} 
\frac{(r+x-1)^{n+\alpha}(1+r+x)^{n+\beta}}{r^{n}}
=\frac{|\sin(\pi\alpha)|}{\pi\alpha} e^{nf(1)}
\end{aligned}
$$
completing the proof for Case 1. 

Case 2 is trivial
since $1$ and $-1$ are both outside $C(x,r)$.
For the last case
we observe that
$$
I^{(\alpha,\beta)}_n(x)=
(-1)^nI^{(\beta,\alpha)}_n(-x)
$$
and likewise
$$
J^{(\alpha,\beta)}_n(x)=
(-1)^nJ^{(\beta,\alpha)}_n(-x).
$$
Moreover, from (\ref{defi f}) we see that 
$f_{b,a,-x}(t)=f_{a,b,x}(-t)$. Now Case 3 follows easily from Case 1.
\qed
\end{proof}

\begin{Lemma}\label{est f(1)} 
Let $t_0\in (t_1,t_2)$ be given by {\rm (\ref{defi t0})}. 
Then
$$
f(1)\leq f(t_0)+\frac1{140} f''(t_0),
$$
in Case 1, and likewise, in Case 3,
$$
f(-1)\leq f(t_0)+\frac1{140} f''(t_0).
$$
\end{Lemma}

\begin{proof} 
It follows from (\ref{defi t0}) that the derivative 
of $t_0=t_0(x)$ as a function of $x$ is
$$\frac{-(a+b)+(b-a)x}{2(a+b+1)^{1/2}(1-x^2)^{3/2}}.$$
Since $|b-a|\le a+b$ it follows that $t_0$ 
is a 
decreasing function of $x\in (-1,1)$. Hence
in Case 1,
$$
t_0(x)<t_0(\frac{a+b}{a+b+2})=
\frac{(b-a)(a+b+2)-(a+b)^2}{4(a+b+1)}\le\frac12$$
where the last inequality follows from 
$$(b-a)(a+b+2)-(a+b)^2= -2a(a+b+1)+2b \le 2(a+b+1).$$
From Proposition \ref{lemma bound f} and
(\ref{eq f''}) we have
$$
f(1)\le f(t_0)+
\frac{(1-t_0)^2}{28(1+t_0)^2}f''(t_0)
$$
with $f''(t_0)< 0$.
Since $t_0\le\frac12$ we find
$$4t_0^2-10t_0+4=4(t_0-\frac12)(t_0-2)\ge 0$$
and
$$
\frac{(1-t_0)^2}{1+t_0^2}-\frac 15
=\frac{4t_0^2-10t_0+4}{5(1+t_0^2)}\ge 0.
$$
Hence
$$
f(1)\le f(t_0)+\frac1{140} f''(t_0)
$$
as claimed. The proof in Case 3 follows by the observation
at the end of the proof of Lemma \ref{remainder estimates},
since the $t_0$ associated with the data $b,a,-x$ is
the negative of the $t_0$ associated with $a,b,x$.
\qed\end{proof}

We can now complete the proof of Theorem \ref{thm1}.
As in (\ref{I ineq}) we find
$$
|J^\ab_n(x)|
\le \frac1{\pi} \int_{0}^{\pi}
e^{n f(\cos\theta)}\,d\theta
\le C_ 1\, e^{n f(t_0)} (n|f''(t_0)|)^{-1/4}
$$
where $C_1=2D^{-1/4}=2\sqrt[4]{28}$.
Since $e^{-t}\le \frac1{\sqrt2}t^{-1/4}$ 
for all $t>0$ 
we obtain from Lemmas 
\ref{remainder estimates} and \ref{est f(1)}
that
$$
|R^\ab_n(x)|\le C_2\, e^{n f(t_0)} (n|f''(t_0)|)^{-1/4}
$$
with $C_2=\frac1{\sqrt2}\sqrt[4]{140}=
\sqrt[4]{35}$. All together
$$
|I^\ab_n(x)|\le C_3 \,e^{n f(t_0)} (n|f''(t_0)|)^{-1/4}
$$
with $C_3=C_1+C_2$.
Still proceeding as in Section \ref{Section integral}
and using Lemma \ref{Stirling bound},
we finally get
$$
\begin{aligned}
|g^\ab_n(x)|&\le C_3
\left(\frac{(n+1)(n+\alpha+\beta+1)}{(n+\alpha+1)(n+\beta+1)}\right)^{1/4}
\left(n|f''(t_0)|\right)^{-1/4}\\
&\le C(1+\alpha+\beta+2n)^{-1/4}(1-x^2)^{-1/4}
\end{aligned}
$$
for $C=\sqrt[4]{6}C_3$. In particular, we find $C< 12$.
\qed

\end{document}